\documentclass[12pt]{amsart}
\usepackage{amsfonts}
\usepackage{amssymb}
\usepackage{enumerate}
\usepackage{graphicx}
\usepackage[dvipsnames]{xcolor}

\makeatletter
\@namedef{subjclassname@2010}{  \textup{2010} Mathematics Subject Classification}
\makeatother
\newtheorem{theorem}{Theorem}[section]
\newtheorem{lemma}[theorem]{Lemma}
\newtheorem{proposition}[theorem]{Proposition}
\newtheorem{corollary}[theorem]{Corollary}

\theoremstyle{definition}

\numberwithin{equation}{section}
\DeclareMathOperator{\diam}{diam}

\newcommand{\R}{\mathbb{R}}
\newcommand{\N}{\mathbb{N}}

\DeclareMathOperator{\ch}{ch}
\newcommand{\Fixx}{\mathrm{Fix}}
\makeatletter
\@namedef{subjclassname@2020}{  \textup{2020} Mathematics Subject Classification}
\makeatother

\begin{document}
\title[On some results related to the Karlsson-Nussbaum conjecture in geodesic spaces]{On some results related to the Karlsson-Nussbaum conjecture in geodesic spaces}
\author[A. Huczek]{Aleksandra Huczek}
\address{Department of Mathematics, Pedagogical University of Krakow,
PL-30-084 Cracow, Poland}
\email{aleksandra.huczek@up.krakow.pl}
\date{}

\begin{abstract}
We show a Wolff-Denjoy type theorem in the case of a one-parameter continuous semigroups of nonexpansive mappings in which there is a compact mapping. Using the notion of attractor we are also able to prove some specific properties directly related to the Karlsson-Nussbaum conjecture.
\end{abstract}

\maketitle

\section{Introduction}

Let us suppose that $f:\Delta \rightarrow
\Delta $ is a holomorphic map of the unit disc $\Delta \subset \mathbb{C}$
without a fixed point. The classical Wolff-Denjoy theorem asserts that then there is a point $\xi \in \partial \Delta $ such that the iterates $f^{n}$ converge locally uniformly to $\xi $ on $\Delta $. There is a wide literature concerning various generalizations of this theorem (see \cite{Ab, Bu1, Ka1, Nu, LLNW, Pi} and references therein). One of the important generalizations was given by Beardon who initially considered the Wolff-Denjoy theorem in a purely geometric way depending only on the hyperbolic properties of a metric \cite{Be1}. In the next step Beardon developed his results for strictly convex bounded domains with the Hilbert metric \cite{Be2}. Considering the notion of the omega limit set $\omega _{f}(x)$ as the set of accumulation points of the sequence $f^{n}(x)$ and the notion of the attractor $\Omega _{f}=\bigcup_{x\in D}\omega _{f}(x)$, Karlsson and Noskov showed in \cite{KaNo} that if a bounded convex domain $D$ is endowed with the Hilbert metric $d_{H}$, then the attractor $\Omega _{f}$ of a fixed-point free nonexpansive map $f:D\rightarrow D$ is a star-shaped subset of $\partial D$. This has led to a conjecture formulated by Karlsson and Nussbaum (see \cite{Ka3, Nu}) asserting that if $D$ is a bounded convex domain in a finite-dimensional real vector space and $f:D\rightarrow D$ is a fixed point free nonexpansive mapping acting on the Hilbert metric space $(D,d_{H})$, then there exists a convex set $\Omega \subseteq \partial D$ such that for each $x\in D$, all accumulation points $\omega _{f}(x)$ of the orbit $O(x,f)$ lie in $\Omega $. It remains one of the major problems in the field.

The first objective of this paper is to extend the Wolff-Denjoy type theorem to the case of one-parameter continuous semigroups of nonexpansive ($1$-Lipschitz) mappings in some quasi-geodesic spaces.  
We show in Section 3 that if $(Y,d)$ is a $(1,\kappa)$-quasi-geodesic space satisfying Axiom $1'$ and Axiom $4'$ (introduced in Section 2) and $S=\{f_{t}:Y\rightarrow Y\}_{t\geq 0}$ is a one-parameter continuous semigroup of nonexpansive mappings on $Y$ without bounded orbits and there exists $t_0>0$ such that $f_{t_0}:Y \rightarrow Y$ is a compact mapping, then there exists $\xi \in \partial Y$ such that the semigroup $S$ converges uniformly on bounded sets of $Y$ to $\xi$. We apply our consideration in the case of bounded strictly convex domains in a Banach space with the metric satisfying condition 
\begin{equation}
	d(sx+(1-s)y,z)\leq \max \{d(x,z),d(y,z)\}.
	\tag{C}
	\label{C}
\end{equation}  
In particular, we obtain the Wolff-Denjoy type theorem for Hilbert’s and Kobayashi’s metrics but its applicability is much wider.

The second aim of this note is to show some results related to the Karlsson-Nussbaum conjecture. We start by generalizing the Abate and Raissy result asserting that a big horoball considered in a bounded and convex domain with the Kobayashi distance is a star-shaped set with respect to the center of the horoball. Using this fact we show in Section 4 that for a compact nonexpansive mapping $f$ acting on a bounded convex domain in a Banach space equipped with the metric satisfying condition ($\ref{C}$), there exists $\xi$ on the boundary such that $\Omega_{f} \subset \ch (\ch(\xi))$. Karlsson proved in \cite{Ka3} a special case of the Karlsson-Nussbaum conjecture by showing that the attractor of a fixed point free nonexpansive mapping is a star-shaped subset of the boundary of the space. In the proof he used the Gromov product. We present a shorter proof of a little more general result, without using the Gromov product, for every metric space $(D,d)$ satisfying condition 
\begin{equation*}
	\lim_{n\rightarrow \infty } [d(x_{n},y_{n})-\max
	\{d(x_{n},z),d(y_{n},z)\}]=\infty 
\end{equation*} 
for any sequences $\{x_n\}$ and $\{y_n\}$ converging to distinct points $x$ and $y$ on the boundary such that the segment $[x,y] \nsubseteq \partial D$, and for any $z \in D$ (see Section 4 for more details).

In Section 5 we formulate counterparts of the above results for a one-parameter continuous semigroup of nonexpansive mappings.

\section{Preliminaries}

Let $(Y, d)$ be a metric space. Recall
that a curve $\gamma :[a,b]\rightarrow Y$ is called 
\textit{$(\lambda ,\kappa)$-quasi-geodesic} if there exists $\lambda \geq 1,\kappa \geq 0$ such
that for all $t_{1},t_{2}\in \lbrack a,b],$
\[
\frac{1}{\lambda }\left\vert t_{1}-t_{2}\right\vert -\kappa \leq d(\gamma
(t_{1}),\gamma (t_{2}))\leq \lambda \left\vert t_{1}-t_{2}\right\vert
+\kappa .
\]
A metric space $Y$ is called $(\lambda ,\kappa)$\textit{-quasi-geodesic} if
every pair of points in $Y$ can be connected by a $(\lambda ,\kappa)$-quasi-geodesic. A convex bounded domain in a Banach space equipped with the Kobayashi distance is an example of a $(1, \kappa)$-quasi geodesic metric space for any $\kappa>0$. 

We are interested in considerations related to nonexpansive and contractive mappings in $(\lambda ,\kappa)$-quasi-geodesic metric spaces. We call a map $f:Y\rightarrow Y$ \textit{nonexpansive} if $$d(f(x),f(y))\leq d(x,y)$$ for every $x,y\in Y$. A map $f:Y\rightarrow Y$ is called \textit{contractive }if $$d(f(x),f(y))<d(x,y)$$ for every distinct $x,y\in Y$. 

Beardon \cite{Be2} and Karlsson \cite{Ka1} considered some four properties of metric spaces called axioms. Proceeding similarly, we slightly modify these axioms to make them better suited for non-proper metric spaces.\\

\noindent \textbf{Axiom 1'}. The metric space $(Y,d)$ is an open dense
subset of a metric space $(\overline{Y},\overline{d})$, whose relative
topology coincides with the metric topology. For any $w\in Y$, if $\{x_{n}\}$
is a sequence in $Y$ converging to $\xi \in \partial Y=\overline{Y}\setminus
Y$, then $$d(x_{n},w)\rightarrow \infty. $$ 
%(the compactness of $\overline{Y}$ is not required).\medskip

%The above property is equivalent to properness of $Y$, that is, every closed and bounded subset is compact. Moreover, 
Notice that Axiom 1' implies that if $A\subset Y$ is bounded, then the $
\overline{d}$-closure of $A$ does not intersect the boundary $\partial Y$
and hence coincides with the $d$-closure of $A$. \\

\noindent \textbf{Axiom 3'}. If $\{x_{n}\}$ and $\{y_{n}\}$ are sequences in
$Y$, $x_{n}\rightarrow \xi \in \partial Y$, $y_{n}\rightarrow \eta \in
\partial Y$, and if for some $w\in Y,$
\[
d(x_{n},y_{n})-d(y_{n},w)\rightarrow -\infty ,
\]
then $\xi =\eta .$\medskip \\

\noindent \textbf{Axiom 4'}. If $\{x_{n}\}$ and $\{y_{n}\}$ are sequences in
$Y$, $x_{n}\rightarrow \xi \in \partial Y$, $y_{n}\rightarrow \eta \in
\partial Y$, and if for all $n,$
\[
d(x_{n},y_{n})\leq c
\]
for some constant $c$, then $\xi =\eta .$\medskip

It is not hard to show that Axioms $1'$ and $2'$ imply Axiom $3'$ and Axioms $1'$ and $3'$ imply Axiom $4'$. We say that a mapping  $f:Y\rightarrow Y$ is \textit{compact}
if $\overline{f(Y)}^{\overline{d}}$, the $\overline{d}$-closure of $f(Y)$, is compact in $(\overline{Y},\overline{d}).$ \\

Let us introduce one more property of metric spaces satisfying Axiom 1' that we will call Axiom $5'$. \\

\noindent \textbf{Axiom 5'}
If $\{x_n\}$ and $\{y_n\}$ are sequences in $Y$ and $d(x_n,y_n) \rightarrow 0$, as $n \to \infty$, then 
\begin{equation*}
	\overline{d}(x_n,y_n) \rightarrow 0, \quad n \to \infty.
\end{equation*}

The following notion of a horoball will be needed throughout the paper. We recall the general definitions introduced by Abate in \cite{Ab}. Define the \textit{small
	horoball} of center $\xi \in \partial Y$, pole $z_{0}\in Y$ and radius $r\in
\mathbb{R}$ by
\[
E_{z_{0}}(\xi ,r)=\{y\in Y:\limsup_{w\rightarrow \xi }d(y,w)-d(w,z_{0})\leq
r\}
\]
and the \textit{big horoball} by
\[
F_{z_{0}}(\xi ,r)=\{y\in Y:\liminf_{w\rightarrow \xi }d(y,w)-d(w,z_{0})\leq
r\}.
\]

Karlsson in \cite{Ka1} and Lemmens and Nussbaum in \cite{LeNu} note that in a metric space satisfying Axioms $1$ and $2$, each horoball intersects the boundary of the space at exactly one point. The next lamma shows an important property of big horoballs i.e. Axiom $3'$ can be regarded as a characteristic that the intersection of horoball's closure consists of a single point.

\begin{lemma}
	\label{a3} Let $(Y,d)$ satisfy Axiom $3'$, $z_{0}\in Y$, $\zeta \in \partial Y$
	and $r\in \mathbb{R}$. Then $$\bigcap_{r\in \mathbb{R}}\overline{
		F_{z_{0}}(\zeta ,r)}=\{\zeta \}.$$
\end{lemma}
The proof proceeds similarly to the finite-dimensional case that can be found in \cite{HuWi1}. \\

We conclude this section by recalling the definitions of Hilbert's and
Kobayashi's metrics. 
Let $K$ be a
closed normal cone with a non-empty interior in a real Banach space $V$. We
say that $y\in K$ \textit{dominates} $x\in V$ if there exists $\alpha ,\beta
\in \mathbb{R}$ such that $\alpha y\leq x\leq \beta y$. This notion yields
on $K$ an equivalence relation $\sim _{K}$ by $x\sim _{k}y$ if $x$ dominates
$y$ and $y$ dominates $x$. For all $x,y\in K$ such that $x\sim _{K}y$ and
$y\neq 0,$ define
\[
M(x/y)=\inf \{\beta >0:x\leq \beta y\}
\]
and
\[
m(x/y)=\sup \{\alpha >0:\alpha y\leq x\}.
\]
The\textit{\ Hilbert (pseudo-)metric} is defined by
\[
d_{H}(x,y)=\log \bigg(\frac{M(x/y)}{m(x/y)}\bigg).
\]
Moreover, we put $d_{H}(0,0)=0$ and $d_{H}(x,y)=\infty
$ if $x\nsim _{K}y$. It can be shown that $d_{H}$ is a metric iff $x=\lambda y$ for some $
\lambda >0.$ 

It will be important to us that a bounded convex domain in a Banach space equipped with the Hilbert metric is a complete geodesic space that satisfies Axiom $1'$ [see \cite{Nu}, Theorem 4.13]). Moreover, the following theorem is true.

\begin{theorem}
	Let $D$ be a bounded convex domain in a real Banach space and let $\{x_n\},\{y_n\} \subset D$. If $d_H(x_n,y_n) \rightarrow 0$, as $n \to \infty$, then 
	\begin{equation*}
		||x_n-y_n|| \rightarrow 0, \quad n \to 0.
	\end{equation*}
\end{theorem}

\begin{proof}
	Let $\{x_n\},\{y_n\} \subset D$ be sequences in $D$ and for any $n \in \N$ consider a straight line passing through $x_n$ and $y_n$ that intersects the boundary of $D$ in precisely
	two points $a_n$ and $b_n$ such that $x_n$ is between $a_n$ and $y_n$, and $y_n$ is between $x_n$ and $b_n$.
	Then 
	\begin{eqnarray*}
		d_H(x_n,y_n) &=& \log \bigg( \frac{||y_n-a_n ||\cdot||x_n-b_n||}{||x_n-a_n||\cdot||y_n-b_n||} \bigg) \\
		&=& \log \bigg( \frac{(||x_n-a_n||+||x_n+y_n||)(||y_n-b_n||+||x_n-y_n||)}{||x_n-a_n||\cdot||y_n-b_n||} \bigg) \\
		&=& \log \bigg(\bigg(1+\frac{||x_n-y_n||}{||x_n-a_n||}\bigg)\bigg(1 + \frac{||x_n-y_n||}{||y_n-b_n||}\bigg)   \bigg).
	\end{eqnarray*}	
	
	Let $d = \diam D$. We note that $||x_n-a_n||\leq d$ and $||y_n-b_n||\leq d$ for every $n \in \N$. Hence 
	\begin{equation*}
		\bigg(1+\frac{||x_n-y_n||}{||x_n-a_n||}\bigg)\bigg(1 + \frac{||x_n-y_n||}{||y_n-b_n||}\bigg) \geq \bigg(1+\frac{||x_n-y_n||}{d}\bigg)\bigg(1 + \frac{||x_n-y_n||}{d}\bigg).
	\end{equation*}
	Consider subsequences $\{x_{n_k}\}$ and $\{y_{n_k}\}$ of sequences $\{x_{n}\}$ and $\{y_{n}\}$, respectively. If there existed  $\delta>0$ and $k_0 \in \N$ such that $||x_{n_k}-y_{n_k}|| \geq \delta$ for $k \geq k_0$, then 
	\begin{equation}\label{delta1} 
		\bigg(1+\frac{||x_{n_k}-y_{n_k}||}{d}\bigg)\bigg(1 + \frac{||x_{n_k}-y_{n_k}||}{d}\bigg) \geq \bigg(1+\frac{\delta}{d}\bigg)\bigg(1 + \frac{\delta}{d}\bigg) = \delta_1 >1
	\end{equation}
	that would contradict our assumption $d_H(x_n,y_n) \rightarrow 0$. Therefore $||x_n-y_n|| \rightarrow 0$, as $n \to \infty$.
\end{proof}

It means that a convex bounded domain in a Banach space equipped with the Hilbert metric satisfies Axiom $5'$. \\

Recall that if $D$
is a bounded convex domain of a complex Banach space $V$, then the \textit{Kobayashi distance} between $z,w\in D$ is given by
\[
k_{D}(z,w)=\inf \{k_{\Delta }(0,\gamma )\mid \exists \varphi \in \mbox{Hol}
(\Delta ,D):\varphi (0)=z,\,\varphi (\gamma )=w\},
\]
where $k_{\Delta }$ denotes the Poincar\'{e} metric on the unit disc $\Delta
$.

It is well known that if $D$ is a bounded and convex domain in a complex Banach space, then $(D,k_D)$ is a complete metric space satisfying Axiom $1'$ (see \cite{Ha}, \cite{KRS}). \\

The second example of a space that satisfies the Axiom $5'$ (with respect to the closure in the norm $(\overline{D},||\cdot||)$) is a convex bounded domain in a Banach space equipped with the Kobayashi distance (see Theorem 3.4, \cite{KRS}). It follows from the following theorem.

\begin{theorem}
	Suppose $D$ is a bounded convex domain of a Banach space. Then  
	\begin{equation*}
		\arg \tanh \bigg(\frac{||x-y||}{\diam_{||\cdot||}D}\bigg) \leq k_D(x,y)
	\end{equation*}
	for every $x,y \in D$.
\end{theorem}

 It turns out that both Hilbert's and Kobayashi's metric $d$ satisfies the following condition ($C$) that is equivalent to the convexity of balls in $D$:
\begin{equation*}
	\tag{C}
	\label{C}
	d(sx+(1-s)y,z)\leq \max \{d(x,z),d(y,z)\}
\end{equation*}%
for all $x,y,z\in D$ and $s\in \lbrack 0,1]$ (see, e.g., \cite{HuWi1}). 

\section{Wolff-Denjoy theorems for semigroups}

Although, the subject of this section is the Wolff-Denjoy theorem in the continuous case (i.e., for semigroups of nonexpansive mappings), we will need the Wolff-Denjoy type theorem for iterates of a nonexpansive mapping (see \cite{HuWi1}).

\begin{theorem} \label{collinf} 
	Let $(Y,d)$ be a $(1,\kappa)$-quasi-geodesic space satisfying Axiom $1'$ and Axiom $4'$, and suppose that for every  $\zeta \in \partial Y$ and $z_{0}\in Y$, the intersection of horoballs' closures $\bigcap_{r\in \mathbb{R}}\overline{F_{z_{0}}(\zeta ,r)}^{\overline{d}}$ consists of a single point. If $f:Y \rightarrow Y$ is a compact nonexpansive mapping without bounded orbits, then there exists $\xi \in \partial Y$ such that the iterates $f^n$ of $f$ converge uniformly on bounded sets of $Y$ to $\xi$. 
\end{theorem}

In the further considerations we will use the following lemma, the proof of which can be found in \cite{HuWi2}.

\begin{lemma}\label{compactK} Let $S=\{f_{t}:Y\rightarrow Y\}_{t\geq 0}$ be a one-parameter
	continuous semigroup of nonexpansive mappings on $Y$, $C$ a compact
	subset of $Y$ and $t_{0}>0.$ Then 
	\begin{equation*}
		K=\{f_{s}(x):0\leq s\leq t_{0},\;x\in C\}=\bigcup_{0\leq s\leq t_{0}}f_{s}(C)
	\end{equation*}%
	is compact.	
\end{lemma}

One of the classical arguments in this line of research is the Ca\l ka theorem which in the original version concerns a metric space with the property that
each bounded subset is totally bounded. We will need the following counterpart of the Ca\l ka theorem for semigroups (\cite{HuWi2}, Theorem 3.3). 

\begin{theorem}
	\label{calkaC}Suppose that $(Y,d)$ is a proper metric space. Let $x_{0}\in Y$
	and $S=\{f_{t}:Y\rightarrow Y\}_{t\geq 0}$ be a one-parameter continuous
	semigroup of nonexpansive mappings. If there exists a sequence of real
	numbers $\{s_{k}\}$ such that a sequence $\{f_{s_{k}}(x_{0})\}$ is bounded,
	then the orbit $O(x_{0})=\{f_{t}(x_{0}):t\geq 0\}$ of $S$ is bounded.
\end{theorem}  

Now, we can prove the main results of this section. 

\begin{theorem} \label{main_inftycomp}
Let $(Y,d)$ be a $(1,\kappa)$-quasi-geodesic space satisfying Axiom $1'$ and Axiom $4'$, and suppose that for every  $\zeta \in \partial Y$ and $z_{0}\in Y$ the intersection of horoballs' closures $\bigcap_{r\in \mathbb{R}}\overline{F_{z_{0}}(\zeta ,r)}^{\overline{d}}$ consists of a single point. If $S=\{f_{t}:Y\rightarrow Y\}_{t\geq 0}$ is a one-parameter continuous semigroup of nonexpansive mappings on $Y$ without bounded orbits, and there exists $t_0>0$ such that $f_{t_0}:Y \rightarrow Y$ is a compact mapping, then there exists $\xi \in \partial Y$ such that the semigroup $S$ converges uniformly on bounded sets of $Y$ to $\xi$.
\end{theorem}

\begin{proof}
	Fix $t_0>0$. We choose a bounded set $D \subset Y$, and define the set 
	\begin{equation*}
	K=\{f_{s}(x):0\leq s\leq t_{0},\;x\in D\}=\bigcup_{0\leq s\leq
		t_{0}}f_{s}(D)\subset Y.
	\end{equation*} 
We note that $K$ is bounded. Indeed, since $D$ is bounded, there exist $x_0 \in Y$ and $r>0$ such that $D \subset B(x_0,r)$. By nonexpansivity of $f_s$ we get $f_s(D) \subset B(f_s(x_0),r)$ for any $s>0$. It follows from boundedness of $D$ that
	\begin{equation*}
	K = \bigcup_{0\leq s\leq t_{0}}f_{s}(D) \subset \bigcup_{0\leq s\leq t_{0}} B(f_s(x_0),r).
	\end{equation*} 
 It follows from Lemma \ref{compactK} that the set $\{f_s(x_0) : \, 0 \leq s \leq t_0 \}$ is compact and hence bounded. Then there exist $y_0 \in Y$ and $r_1>0$ such that $\{f_s(x_0) : \, 0 \leq s \leq t_0 \} \subset B(y_0,r_1).$ Hence  $
	K = \bigcup_{0\leq s\leq t_{0}} B(f_s(x_0),r) \subset B(y_0,r+r_1)$
is bounded. Fix $y \in Y$ and note that $
	f_{nt_0}(y) = f^n_{t_0}(y)$ for any $n \in \N$.
	Furthermore, $f_t(y) = f_{t_0}(f_{t - t_0}(y)) \in \{\overline{f_{t_0}(Y)}\}^{\overline{d}}$ for any $t \geq t_0$. Since the map  $f_{t_0}$ is compact, the $\overline{d}$-closure of the orbit $O(y)= \{f_{t}(y): t \geq t_0\}$ is a compact subset of $\overline{Y}$.  Next, consider a $d$-closed and bounded set $B \subset Y$. It follows from Axiom $1'$ that  $B$ is $\overline{d}$-closed. Then the set $
	\overline{O(y)}^{\overline{d}} \cap B = \overline{O(y)}^{d} \cap B
	$ is compact in  $\overline{Y}$ and hence compact in $Y$.
That is, $(\overline{O(y)}^{d},d)$ is proper, 
$f_t: \overline{O(y)}^{d} \to \overline{O(y)}^{d}$ for $t \geq 0$, and by  Całka's theorem and nonexpansivity of $f_{t_0}$ we get  $$d(f_{nt_0}(y),y)\rightarrow \infty, \quad n\rightarrow \infty. $$ It follows from Theorem \ref{collinf} that
	\begin{equation}\label{zbcomp}
	\sup_{y\in K}\,\overline{d}(f_{nt_{0}}(y),\xi )\rightarrow 0, 
	\end{equation}%
as $n\rightarrow \infty $. We note that for every $x\in C$, $t>0$ such that $t = nt_0+s$ and  $n\in \mathbb{N}$, $s\in \lbrack 0,t_{0})$, we have	\begin{equation*}
	f_{t}(x)=f_{nt_{0}+s}(x)=f_{nt_{0}}(f_{s}(x))\in f_{nt_{0}}(K).
	\end{equation*}%
	Therefore by  (\ref{zbcomp}), 
	\begin{equation*}
	\sup_{x\in C}\overline{d}(f_{t}(x),\xi )=\sup_{x\in C}\overline{d}%
	(f_{nt_{0}}(f_{s}(x)),\xi )\leq \sup_{y\in K}\overline{d}(f_{nt_{0}}(y),\xi
	)\rightarrow 0.
	\end{equation*}%
It follows that $\sup_{x\in C}\overline{d}(f_{t}(x),\xi )\rightarrow 0,$ as $t\rightarrow \infty$. 
\end{proof}

 Lemma \ref{a3}, combining with Theorem \ref{main_inftycomp}, immediately leads to the following corollary.

\begin{corollary}
Let $(Y,d)$ be a $(1,\kappa)$-quasi-geodesic space satisfying Axiom $1'$ and Axiom $3'$. If  $S=\{f_{t}:Y\rightarrow Y\}_{t\geq 0}$ is a one-parameter continuous semigroup of nonexpansive mappings on $Y$ without bounded orbits, and there exists $t_0>0$ such that the mapping $f_{t_0}:Y \rightarrow Y$ is compact, then there exists $\xi \in \partial Y$ such that the semigroup $S$ converges uniformly on bounded sets of $Y$ to $\xi$.
\end{corollary}

The next theorem shows that if there exists in the semigroup $S=\{f_{t}:Y\rightarrow Y\}_{t\geq 0}$ a compact and contractive mapping $f_{t_{0}}:Y\rightarrow Y$, then we get a slightly stronger result.

\begin{theorem} Let $(Y,d)$ be a $(1,\kappa)$-quasi-geodesic space satisfying Axiom $1'$ and Axiom $4'$, and suppose that for every  $\zeta \in \partial Y$ and $z_{0}\in Y$, the intersection of horoballs' closures $\bigcap_{r\in \mathbb{R}}\overline{F_{z_{0}}(\zeta ,r)}^{\overline{d}}$ consists of a single point. If $S=\{f_{t}:Y\rightarrow Y\}_{t\geq 0}$ is a one-parameter continuous semigroup of nonexpansive mappings on $Y$, and there exists $t_{0}>0$ such that a mapping $f_{t_{0}}:Y\rightarrow Y$ is compact and contractive, then there exists $\xi \in \overline{Y}$ such that the semigroup $S$ converges uniformly on compact sets of $Y$ to $\xi$.
\end{theorem}

\begin{proof}
	In the first step, suppose that the semigroup  $S=\{f_{t}:Y \rightarrow Y\}_{t\geq 0}$ has unbounded orbits. Then the conclusion follows directly from Theorem \ref{main_inftycomp}. Therefore, we assume that $\{f_{t}(y)\}_{t\geq 0}$ is bounded for every $y\in Y$. By assumption, there exists $t_{0} >0$ such that  $f_{t_{0}}:Y\rightarrow Y$ is a compact and contractive mapping. Fix  $y_0 \in Y$.  Note that$ f_{nt_0}(y_0) = f^n_{t_0}(y_0)$, $n \in \N$.
	Since the mapping $f_{t_0}$ is compact and
	$\{f_{nt_0}(y_0)\}_{n\in \N}$ is bounded
    we have that $\overline{\{f_{nt_0}(y_0)\}}^{\overline{d}}
    _{n\in \N}    
    =\overline{\{f_{nt_0}(y_0)\}}^{{d}}_{n\in \N}$  is compact in $\overline{Y}$ and hence also in $Y$. It follows that there exists a subsequence $\{f_{n_{k}t_{0}}(y_0)\}$ of $\{f_{nt_{0}}(y_0)\}$ converging to some $z_0 \in Y$. Since $f_{t_0}$ is nonexpansive, the sequence	
	$$
	d_{n}=d(f_{nt_{0}}(y_0),f_{nt_{0}+t_0}(y_0)), \; \; n=1,2,\ldots. 
	$$
	is nonincreasing and hence it converges to some $\eta $, as $n\rightarrow \infty $. Therefore,
	\begin{equation*}
	\eta \leftarrow
	d_{n_{k}}=d(f_{n_{k}t_{0}}(y_0),f_{n_{k}t_{0}+t_0}(y_0))\rightarrow d(z_0,f_{t_0}(z_0))
	\end{equation*}%
	and
	\begin{equation*}
	\eta \leftarrow
	d_{n_{k}+1}=d(f_{n_{k}t_{0}+t_{0}}(y_0),f_{n_{k}t_{0}+t_{0}+t_0}(y_0))\rightarrow
	d(f_{t_{0}}(z_0),f_{2t_{0}}(z_0)).
	\end{equation*}
	Hence 
	\begin{equation}\label{convtoetaa} 
	\eta =d(z_0,f_{t_0}(z_0))=d(f_{t_{0}}(z_0),f_{2t_{0}}(z_0)).
	\end{equation}
	Since $f_{t_{0}}$ is contractive, if $z_0$ and $f_{t_0}(z_0)$ were distinct points, we would have $$d(f_{t_{0}}(z_0),f_{2t_{0}}(z_0))<d(z_0,f_{t_0}(z_0)),$$
	a contradiction with (\ref{convtoetaa}). Thus $f_{t_0}(z_0)=z_0.$ Moreover, since the sequence $\{d(f_{nt_{0}}(y_0),z_0)\}$ is decreasing
	and $f_{n_{k}t_{0}}(y_0)\rightarrow z_0\in Y$, we have also $f_{nt_{0}}(y_0)\rightarrow z_0,$ as $n\rightarrow \infty$. Now, notice that if we choose any $y \in Y$, then the previous reasoning shows that $f_{nt_0}(y)$ converges to a fixed point of $f_{t_0}$ for every $y \in Y$. However, a contractive mapping has at most one fixed point. Therefore  $f_{nt_{0}}(y)\rightarrow z_0,$ as $n \to \infty$, for all $y \in Y$. \\
	
Now we show uniform convergence on compact sets. We choose a compact set $C \subset Y$ and define the set $K$ as in the proof of Theorem \ref{main_inftycomp}:
\begin{equation*}
	K = \{f_s(x) : 0 \leq s \leq t_0, \, x \in C \}.
\end{equation*}
Fix $\varepsilon >0$. Since by Lemma \ref{compactK} the set $K$ is compact, note that for some $y_{1},\ldots,y_{n}\in K$, $$K\subset \bigcup_{i=1}^{n}B \Big(y_{i},\frac{\varepsilon }{2}\Big).$$
From the first part of the proof, there exists $n_{0}$ such that for any $n\geq n_{0},$ $$\sup_{i=1,\ldots ,n}d(f_{nt_{0}}(y_{i}),z_0)<%
\frac{\varepsilon }{2}.$$ We choose $y\in K$, then there exists $i$ such that
$d(y,y_{i})<\frac{\varepsilon }{2}$. By nonexpansivity  of $f_{nt_0}$ we get   
\begin{align*}
	d(f_{nt_{0}}(y),z_0)& \leq
	d(f_{nt_{0}}(y),f_{nt_{0}}(y_{i}))+d(f_{nt_{0}}(y_{i}),z_0) \\
	& \leq d(y,y_{i})+d(f_{nt_{0}}(y_{i}),z_0) \\
	& <\frac{\varepsilon }{2}+\frac{\varepsilon }{2}=\varepsilon .
\end{align*}%
Since $y\in K$ was chosen arbitrarily, 
\begin{equation}\label{concs}
	\sup_{y\in K}d(f_{nt_{0}}(y),z_0)\rightarrow 0,
\end{equation}%
as $n\rightarrow \infty $. Note that for every $x\in C$ and $t>0$ such that $t = nt_0+s$, $n\in \mathbb{N}$, $s\in \lbrack 0,t_{0})$, we have  
\begin{equation*}
	f_{t}(x)=f_{nt_{0}+s}(x)=f_{nt_{0}}(f_{s}(x))\in f_{nt_{0}}(K).
\end{equation*}
Therefore by (\ref{concs}),
\begin{equation*}
	\sup_{x\in C}d(f_{t}(x),z_0)=\sup_{x\in C}d(f_{nt_{0}}(f_{s}(x)),z_0)\leq
	\sup_{y\in K}d(f_{nt_{0}}(y),z_0).
\end{equation*}%
It follows that $\sup_{x\in C}d(f_{t}(x),z_0)\rightarrow 0,$ as $t\rightarrow
\infty $, and the proof is complete.
\end{proof}

By Lemma \ref{a3}, we immediately obtain the following conclusion.  

\begin{corollary}\label{main_coro}
Let $(Y,d)$ be a $(1,\kappa)$-quasi-geodesic space satisfying Axiom $1'$ and Axiom $3'$. If  $S=\{f_{t}:Y\rightarrow Y\}_{t\geq 0}$ is a one-parameter continuous semigroup of nonexpansive mappings on $Y$ and there exists $t_0>0$ such that the mapping $f_{t_0}:Y \rightarrow Y$ is  compact and contractive, then there exists $\xi \in \overline{Y}$ such that the semigroup $S$ converges uniformly on compact sets of $Y$ to $\xi$.
\end{corollary}

Let $V$ be Banach space and $D$ a bounded convex domain of $V$. Consider 
$\partial D = \overline{D} \setminus D$, where $\overline{D}$ denotes the closure of $D$ in the norm topology. We will always assume that $(D,d)$ is a $(1,\kappa)$-quasi geodesic metric space, whose topology coincides
with the norm topology.  Recall that $D\subset V$ is \textit{strictly convex} if for any $z,w\in \overline{D}$
the open segment $$(z,w)=\{sz+(1-s)w:s\in (0,1)\}$$ lies in $D$.

The following lemma was shown in \cite{HuWi1}.

\begin{lemma}
\label{A3'}	Suppose that $D$ is a convex domain of $V$ and $(D,d)$
	satisfies condition (\ref{C}). If $\{x_{n}\}$ and $\{y_{n}\}$ are sequences
	in $D\,$converging to $\xi $ and $\eta $, respectively, in $\partial D$, and
	if for some $w\in D,$
	\[
	d(x_{n},y_{n})-d(y_{n},w)\rightarrow -\infty ,
	\]
	then $[\xi ,\eta ]\subset \partial D.$
\end{lemma}

Recall that $x_{0}$ is called a fixed point of a mapping 
$f:D\rightarrow D$ if $f(x_{0})=x_{0}.$
Define
\begin{eqnarray*}
	\mathrm{Fix}(f) =\{x\in D:f(x)=x\}. %\\
	%	\mathrm{Fix}(f) &=&\{x\in D:f(x)=x\}.
\end{eqnarray*}%

Corollary \ref{main_coro} and Lemma \ref{A3'} now yields the following result.

\begin{theorem}\label{main_stB}
Suppose that $D$ is a bounded strictly convex domain of a Banach space and $(D,d)$ is $(1,\kappa)$-quasi-geodesic space satisfying Axiom $1'$ and condition ($C$). If $S=\{f_{t}:D\rightarrow D\}_{t\geq 0}$ is 
a one-parameter continuous semigroup of nonexpansive mappings on $Y$, and there exists $t_{0}>0$ such that the mapping $f_{t_{0}}:Y\rightarrow Y$ is compact with $\Fixx (f_{t_0}) = \emptyset,$ then there exists $\xi \in \partial D$ such that the semigroup $S$ converges uniformly on bounded sets of $D$ to $\xi$.
\end{theorem}

\begin{proof}
In the first case, if $\{f_t:D \rightarrow D\}_{t \geq 0}$ has unbounded orbits, then the conclusion follows directly from Theorem \ref{main_inftycomp}. Thus we can assume that the orbit $\{f_{t}(y)\}_{t \geq 0}$ is bounded for some (hence for any) $y \in D$. Let $r(\{f_t(y)\}) = \inf_{z \in D} \limsup_{t \to \infty} d(z,f_{t}(y))$, and note that the asymptotic center
	\begin{equation*}
	A = \{x \in D: \limsup_{t \to \infty} d(x,f_t(y))=r(\{f_t(y)\}) \}
	\end{equation*}
is nonempty. Indeed, $A = \bigcap_{\varepsilon >0} A_{\varepsilon}$, where
	\begin{equation*}
	A_{\epsilon} = \{x \in D : \limsup_{t \to \infty} d(x,f_t(y))\leq r(\{f_t(y)\})+ \epsilon \}.
	\end{equation*}
	Since the mapping $f_{t_0}$ is nonexpansive we note that $f_{t_0}(A_{\varepsilon}) \subset A_{\varepsilon}$. What is more,  $A_{\epsilon}$ is bounded and closed with respect to $d$ and with respect to the norm. Since $f_{t_0}$ is compact,
	\begin{equation*}
	\emptyset \neq \bigcap_{\varepsilon >0}\overline{f_{t_0}(A_{\varepsilon })}^{||\cdot||}\subset
	\bigcap_{\varepsilon >0}A_{\varepsilon }=A.
	\end{equation*}
	Furthermore,    
	\begin{equation*}
	f_{t_0}(A) = f_{t_0}(\bigcap_{\epsilon>0}A_{\epsilon}) \subset \bigcap_{\epsilon>0} f_{t_0}(A_{\epsilon}) \subset \bigcap_{\epsilon>0} A_{\epsilon} = A,
	\end{equation*}
	which means that $f_{t_0}(A) \subset A$. The set $A$ is also bounded and closed in $||\cdot||$ and $\overline{f_{t_0}(A)}^{||\cdot||}$ is compact. Since by assumption $D$ is convex, and the metric space $(D,d)$ satisfies condition $C$, $A$ is convex, too. Therefore, it follows from the Schauder fixed-point theorem that $f_{t_0}$ has a fixed point, which
	is a contradiction.
	
\end{proof}

We said before that any bounded and convex domain in a Banach space can be equipped with the Hilbert metric and become a complete geodesic space satisfying Axiom $1'$ and condition $(C)$. Hence and from Theorem \ref{main_stB} we have the following corollary.

\begin{corollary}
	Assume that $D$ is a bounded strictly convex domain in  a Banach space. If $S=\{f_{t}:D\rightarrow D\}_{t\geq 0}$ is a one-parameter continuous semigroup of nonexpansive mappings with respect to the Hilbert metric $d_{H}$, and there exists $t_0 > 0$ such that $\Fixx (f_{t_0}) = \emptyset$, and the mapping $f_{t_0}$ is compact, then there exists $\xi \in \partial D$ such that the semigroup $S$ converges uniformly on bounded sets of $D$ to $\xi$.
\end{corollary}

As discussed in Section 2, the Kobayashi distance satisfies all the conditions to formulate the next corollary.

\begin{corollary}
Assume that $D$ is a bounded strictly convex domain in a complex Banach space. If $S=\{f_{t}:D\rightarrow D\}_{t\geq 0}$ is a one-parameter continuous semigroup of nonexpansive mappings with respect to the Kobayashi distance $k_{D}$, and there exists $t_0 > 0$ such that $\Fixx (f_{t_0}) = \emptyset$, and the mapping $f_{t_0}$ is compact, then there exists $\xi \in \partial D$ such that the semigroup $S$ converges uniformly on bounded sets of $D$ to $\xi$.
\end{corollary}

\section{Attractor of a nonexpansive mapping}

Let $V$ be Banach space and $D \subset V$, $f:D \rightarrow D$ and $y\in D$. Then the set of accumulation points (in the
norm topology) of the sequence $\{f^{n}(y)\}$ is called the \textit{omega
	limit set of }$y$ and is denoted by $\omega _{f}(y)$. In other words,
\begin{equation*}
	\omega_{f}(y)=\{x\in \overline{D}: \exists_{\{n_k\}}
	\text{  an increasing sequence such that}%
	\lim_{k \to \infty} f^{n_k}(y)=x \}.
\end{equation*}%
The \textit{attractor }of $f$ is defined as 
\begin{equation*}
	\Omega _{f}=\bigcup_{y\in D}\omega _{f}(y).
\end{equation*}%
Let $D\subset V$ be a
convex domain. Given $\xi \in \partial D,F\subset \partial D$, set
\begin{eqnarray*}
	\ch (\xi ) &=&\{x\in \partial D:[x,\xi ]\subset \partial D\}, \\
	\ch (F) &=&\bigcup_{\xi \in F} \ch (\xi ).
\end{eqnarray*}

We will need the following lemma (see Lemma 5.2, \cite{HuWi1}).

\begin{lemma}
	\label{lem1_comp}Suppose that $Y$ is a $(1,\kappa )$-quasi-geodesic space
	satisfying Axiom 1' and $f:Y\rightarrow Y$ is a compact nonexpansive mapping
	without a bounded orbit. Then there exists $\xi \in \partial Y$ such that
	for every $z_{0}\in Y$, $r\in \mathbb{R}$ and a sequence of natural numbers $
	\{a_{n}\}$, there exists $z\in Y$ and a subsequence $\{a_{n_{k}}\}$ of $
	\{a_{n}\}$ such that $f^{a_{n_{k}}}(z)\in F_{z_{0}}(\xi ,r)$ for every $k\in
	\mathbb{N}$. Moreover, if $Y$ satisfies Axiom 4', then $\xi \in
	\bigcap_{r\in \mathbb{R}}\overline{F_{z_{0}}(\xi ,r)}.$
\end{lemma}

The next theorem is a generalization of Theorem 4.10 in \cite{HuWi1} which in turn is a generalization of the Abate and Raissy result \cite[Theorem 6]{AbRa}, who proved it for bounded convex domains with the Kobayashi distance.

\begin{theorem}\label{chuog}
	Let $D$ be a bounded convex domain in a Banach space  $V$ and let $(D,d)$ be a complete $(1,\kappa)$-quasi geodesic space satisfying Axiom $1'$ and condition (\ref{C}), whose topology coincides with the norm topology. If $f:D \rightarrow D$ is a compact and nonexpansive mapping without bounded orbits, then there exists $\xi \in \partial D$ such that
	\begin{equation*}
		\Omega_f \subset \ch (\bigcap_{r\in \mathbb{R}} \, \overline{F_{z_0}(\xi,r)}^{||\cdot||})
	\end{equation*}
	for some $z_0 \in D$.
\end{theorem}

\begin{proof}
Fix $y \in D$ and a sequence of natural numbers $\{a_n\}$. Consider a $d$-closed and bounded set $B \subset D$. It follows from Axiom $1'$ that  
%$D$ is $||\cdot||$-closed. Then a set 
$\overline{O(y)}^{\overline{d}} \cap B = \overline{O(y)}^{d} \cap B
$ is compact in  $\overline{D}$ and hence also in $D$.
That is, $(\overline{O(y)}^{d},d)$ is proper and by the Całka theorem \cite[Theorem 2.1]{HuWi1} and nonexpansivity of $f$ we get  $d(f^n(y),y)\rightarrow \infty$, $n\rightarrow \infty. $ 
Suppose that $f^{a_n}(y) \rightarrow \eta \in \partial D$. It follows from Lemma \ref{lem1_comp} that we can choose $\xi \in \partial D$  and fix  $z_0 \in D$ and $r \in \R$ such that there exist $z \in D$ and a subsequence $\{a_{n_k}\}$ of $\{a_n\}$ for which $
		f^{a_{n_k}}(z) \in F_{z_0}(\xi, r)$,
 $k \in \N$. 
	Without loss of generality, we can assume that  
	$$
	f^{a_{n_k}}(z) \rightarrow \zeta_r \in \partial D \cap \overline{F_{z_0}(\xi,r)}^{||\cdot||}, \qquad k \to \infty.
	$$ 
	%Ponieważ 
	%$$
	%d(f^{a_{n_k}}(z),f^{a_{n_k}}(y)) \leq d(z,y),
	%$$
	It follows from Lemma \ref{A3'} that $[\eta, \zeta_r] \subset \partial D$ i.e., for any $r \in \R$ there exists $\zeta_r \in \partial D \cap \overline{F_{z_0}(\xi,r)}^{||\cdot||}$ such that  $\eta \in \ch(\zeta_r)$.
Consider a descending sequence	$\{r_n\}$ such that $r_n \rightarrow - \infty$, as $n \to \infty$. Hence the sequence $\{\zeta_{r_n}\} \subset \partial D$ and the segments of ends in $\eta$ and $\zeta_r$ lie on the boundary. Note that $\zeta_{r_n} \in \overline{f(D)}$ for any $n \in N$. Hence and by compactness of $\partial D \cap \overline{f(D)}$, there is a subsequence $\{r_{n_k}\}$ of $\{r_n\}$ such that $\zeta_{r_{n_k}} \rightarrow \zeta \in \partial D$, as $k \to \infty$. Fix $k_0$. Note that for $k \geq k_0$,  
	\begin{equation*}
		\zeta_{r_{n_k}} \in \overline{F_{z_{0}}(\xi ,r_{n_{k}})}^{||\cdot||} \subset \overline{F_{z_{0}}(\xi ,r_{n_{k_0}})}^{||\cdot||}
	\end{equation*}
and hence
	\begin{equation*}
		\zeta \in \overline{F_{z_{0}}(\xi ,r_{n_{k_0}})}^{||\cdot||}.
	\end{equation*} 
	Since $k_0$ was chosen arbitrarily, 
	\begin{equation*}
		\zeta \in \bigcap_{k\in \mathbb{N}} \overline{F_{z_{0}}(\xi ,r_{n_k})}^{||\cdot||} =  \bigcap_{r\in \mathbb{R}} \overline{F_{z_{0}}(\xi ,r)}^{||\cdot||}.
	\end{equation*}
%	Pozostaje nam pokazać, że $\eta \in \ch (\zeta)$. 
Fix $s \in [0,1]$ and note that
	\begin{equation*}
		||s \zeta_{r_{n_k}} + (1-s)\eta - (s \zeta +(1-s)\eta) || = s ||\zeta_{r_{n_k}} - \zeta || \rightarrow 0,
	\end{equation*}
	as $k \to \infty$, which means that $[\zeta, \eta] \in \partial D$. Therefore 
	\begin{equation*}
		\eta \in \ch (\bigcap_{r\in \mathbb{R}} \, \overline{F_{z_0}(\xi,r)}^{||\cdot||}).
	\end{equation*}
\end{proof}

A big horoball is not always a convex set. However, Abate and Raissy \cite{AbRa} proved that a big horoball considered in a bounded and convex domain with the Kobayashi distance is a star-shaped set with respect to the center of the horoball. We now present a generalization of this fact for all metric spaces satisfying condition (\ref{C}).

\begin{lemma}
	\label{gwC} Let $D$ be a bounded convex domain in a Banach space  $V$ and suppose that $(D,d)$ is $(1,\kappa)$-quasi geodesic space satisfying condition $(C)$, whose topology coincides with the norm topology. If $z_0 \in D$, $r > 0$ and $\xi \in \partial D$, then for every $\eta \in \overline{F_{z_0}(\xi,r)}^{||\cdot||}$ we have
	\begin{equation*}
		[\eta,\xi] \subset \overline{F_{z_0}(\xi,r)}^{||\cdot||}.
	\end{equation*}
\end{lemma}

\begin{proof}
	Fix $\eta \in F_{z_0}(\xi,r)$ and choose a sequence $\{x_n\} \subset D$ converging to $\xi \in \partial D$ and such that the limit 
	\begin{equation*}
		\lim_{n \to \infty} d(\eta,x_n)-d(x_n,z_0) \leq r
	\end{equation*}
	exists. Fix $s \in (0,1)$. 
	It follows from condition (\ref{C}) that
	\begin{equation}\label{lem5C1} 
		\limsup_{n \to \infty} d(s\eta+(1-s)x_n,x_n) - d(x_n,z_0) \leq \lim_{n \to \infty} d(\eta,x_n) - d(x_n,z_0) \leq r.
	\end{equation}
	Note that
	\begin{equation*}
		||s \eta +(1-s)x_n - (s\eta +(1-s)\xi) || = (1-s)||x_n - \xi|| \rightarrow 0, \quad n \to \infty.
	\end{equation*}
	Since topology of $(D,d)$ and $(\overline{D},||\cdot ||)$ coincides on $D$ we get
	\begin{equation*}
		d(s \eta +(1-s)x_n,  s\eta +(1-s)\xi) \rightarrow 0,
	\end{equation*}
	as $n \to \infty$. Hence
	\begin{equation}\label{lem5C2} 
		|d(s\eta+(1-s)x_n,x_n) - d(x_n,s\eta+(1-s)\xi)| \leq d(s\eta+(1 - s)x_n, s\eta+(1-s)\xi)	\rightarrow 0,
	\end{equation}
	as $n \to \infty$. 
	From (\ref{lem5C1}) and (\ref{lem5C2}) we have 
	\begin{eqnarray*}
		\liminf_{w \to \xi} d(s\eta+(1-s)\xi,w) - d(w,z_0) &\leq& \limsup_{n \to \infty} d(s\eta+(1-s)\xi,x_n) - d(x_n,z_0) \\ &\leq& \lim_{n \to \infty} d(s\eta+(1-s)\xi,x_n) - d(x_n,s\eta +(1-s)x_n) \\
		& & + \limsup_{n \to \infty} d(s\eta +(1-s)x_n,x_n) - d(x_n,z_0) \\
		&\leq & r.
	\end{eqnarray*}	
	Hence, for every $s \in (0,1)$, the point $s\eta+(1-s)\xi \in F_{z_0}(\xi,r)$, thus $(\eta,\xi)\subset F_{z_0}(\xi,r)$.
	Therefore,  $[\eta,\xi]\subset \overline{F_{z_0}(\xi,r)}^{||\cdot||}$.
	
To complete the proof, consider the case $\eta \in \partial F_{z_0}(\xi,r)$. So there is such a sequence $\{y_n\} \subset F_{z_0}(\xi,r)$ such that $y_n \rightarrow \eta$, as $n \to \infty$. 
	From the previous considerations we have $[y_n,\xi]\subset \overline{F_{z_0}(\xi,r)}^{||\cdot||}$ for any $n \in \N$. 
	Fix $s \in [0,1]$. Then 
	\begin{equation*}
		||s\eta+(1-s)\xi-(sy_n+(1-s)\xi)|| = s ||\eta - y_n|| \rightarrow 0,
	\end{equation*}
	as $n \to \infty$ which means that $s\eta+(1-s)\xi \in 
	\overline{F_{z_0}(\xi,r)}^{||\cdot||}$. Therefore, $[\eta,\xi] \subset \overline{F_{z_0}(\xi,r)}^{||\cdot||}$, and the proof is complete.
\end{proof}

Theorem \ref{chuog} combined with Lemma \ref{gwC} gives an immediate result.

\begin{theorem}
	\label{chchdys}
	Let $D$ be a bounded convex domain in a Banach space $V$, and suppose that $(D,d)$ is $(1,\kappa)$-quasi geodesic space satisfying Axiom $1'$ and condition (\ref{C}), whose topology coincides with the norm topology. If $f:D \rightarrow D$ is a compact nonexpansive mapping without bounded orbits, then there exists $\xi \in \partial D$ such that 
	\begin{equation*}
		\Omega_f \subset \ch (\ch(\xi)).
	\end{equation*}
\end{theorem}

\begin{proof}
	Fix $z_0 \in D$. It follows from Theorem \ref{chuog} that there exists $\xi \in \partial D$ such that 
	\begin{equation*}
		\Omega_f \subset \ch (\bigcap_{r\in \mathbb{R}} \, \overline{F_{z_0}(\xi,r)}^{||\cdot||}).
	\end{equation*}
	%Z Uwagi \ref{horobound} mamy, że $$\bigcap_{r\in \mathbb{R}} \, \overline{F_{z_0}(\xi,r)}^{||\cdot||} \subset \partial D.$$
	Consider $\zeta \in \bigcap_{r\in \mathbb{R}} \, \overline{F_{z_0}(\xi,r)}^{||\cdot||} $. By Lemma \ref{gwC} we have $	[\zeta,\xi] \subset \overline{F_{z_0}(\xi,r)}^{||\cdot||}$ for any $r \in \R$. Hence  $	 [\zeta,\xi] \subset \bigcap_{r\in \mathbb{R}} \overline{F_{z_0}(\xi,r)}^{||\cdot||}$. Thus $\zeta \in \ch (\xi)$ and hence 
	\begin{equation*}
		\bigcap_{r\in \mathbb{R}} \, \overline{F_{z_0}(\xi,r)}^{||\cdot||} \subset \ch(\xi).
	\end{equation*}
	Therefore, $\Omega_f \subset \ch (\ch(\xi)).$
\end{proof}

Let $D$ be a bounded convex domain in a Banach space $V$ and let $(D, d)$ be a metric space. In the context of Hilbert's metric, Karlsson presented a property which we will call Axiom $2^{\ast}$.\\

\noindent \textbf{Axiom $2^{\ast }$} If $\{x_{n}\}$ and $\{y_{n}\}$ are convergent
sequences in $D$ with limits $x$ and $y$ in $\partial D$, respectively, and
the segment $[x,y]\nsubseteq \partial D$, then for each $z\in D$ we have 
\begin{equation*}
	\lim_{n\rightarrow \infty } [d(x_{n},y_{n})-\max
	\{d(x_{n},z),d(y_{n},z)\}]=\infty .
\end{equation*}

The next lemma shows that the Hilbert metric satisfies Axiom $2^{\ast}$ (see \cite[Theorem 4.13]{Nu}, \cite[Proposition 8.3.3]{LeNu}).

\begin{proposition}
	Let $D \subset V$ be a bounded convex domain in a Banach space. If $\{x_{n}\}$ and $\{y_{n}\}$ are convergent
	sequences in $D$ with limits $x$ and $y$ in $\partial D$, respectively, and
	the segment $[x,y]\nsubseteq \partial D$, then for each $z\in D$ we have 
	\begin{equation*}
		\lim_{n\rightarrow \infty } [d_H(x_{n},y_{n})-\max
		\{d_H(x_{n},z),d_H(y_{n},z)\}]=\infty,
	\end{equation*}
that is, $(D,d_H)$ satisfies Axiom $2^{\ast}$.
\end{proposition}

\begin{proof}
	Consider a sequence $\{u_n\}$ defined as $u_n = \frac{x_n+y_n}{2}$ for any $n \in \N$. Since the segment $[x,y]$ does not lie on the boundary, we get that $u = \frac{x+y}{2} \in D$. Note that 
	\begin{equation*}
	||u_n-u || = \Big|\Big|\frac{x_n+y_n}{2} -\frac{x+y}{2}\Big|\Big| \leq \frac{1}{2} (||x_n-x|| + ||y_n-y||) \rightarrow 0, \quad n \to \infty.
	\end{equation*}
	Since topology of $(D,d_H)$ and $(\overline{D},||\cdot||)$ coincides on $D$ we have 
	$d_H (u_n,u) \rightarrow 0$, as $n \to \infty$. 
	Hence 
	\begin{equation}\label{HA1}
	\lim_{n \to \infty} d_H(u_n, w) = d_H(u,w) < \infty.
	\end{equation} 
	 for any $w \in D$. We note that  
\begin{eqnarray}\label{HA2}
	\begin{split}
	d_H(x_n, y_n) &=& &d_H(x_n, u_n) + d_H(u_n, y_n) \\ &\geq& &d_H(x_n, w) - d_H(u_n, w) +  d_H(y_n, w) - d_H(u_n, w).
	\end{split}
\end{eqnarray}
It follows from Axiom $1'$, (\ref{HA1}) and (\ref{HA2}) that  
\begin{equation*}
	 d_H(x_{n},y_{n})-\max\{d_H(x_{n},z),d_H(y_{n},z)\} \geq \min\{d_H(x_{n},z),d_H(y_{n},z)\} - 2d_H(u_n,w)  \rightarrow \infty,
\end{equation*}
as $n \to \infty$. That means that $(D,d_H)$ satisfies Axiom $2^{\ast}$.
\end{proof}	

In the case of a convex bounded domain D equipped with the Hilbert metric it was proved in \cite{KaNo} that the attractor (in the norm topology) $\Omega_f$ of a fixed point free nonexpansive mapping $f : D \rightarrow D$ is a star-shaped subset of $\partial D$.
Karlsson \cite{Ka3} proved this theorem using Gromov's product and it was important that the geodesics are linear segments. At the end of this section, we present a shorter proof of this theorem for all spaces satisfying Axiom $2^{\ast }$ and not using Gromov's product.

\begin{theorem}
	\label{gwHil} Let $D \subset V$ be a bounded convex domain in a Banach space, and let $(D,d)$ be $(1,\kappa)$-quasi geodesic space satisfying Axiom $1'$ and Axiom $2^{\ast}$, whose topology coincides with the norm topology. Suppose that $f:D \rightarrow D$ is a compact nonexpansive mapping without bounded orbits, then there exists $\xi \in \partial D$ such that  
	\begin{equation*}
		\Omega_f \subset \ch(\xi).
	\end{equation*}
\end{theorem}

\begin{proof}
	Fix $y\in D$ and define a sequence $\{d_n\}$ as $d_{n}=d(f^{n}(y),y)$. Consider a $d$-closed and bounded set $B \subset D$. It follows from Axiom $1'$ that  
	%$D$ is $||\cdot||$-closed. Then a set 
	$\overline{O(y)}^{\overline{d}} \cap B = \overline{O(y)}^{d} \cap B
	$ is compact in  $\overline{D}$ and hence compact in $D$.
	That is, $(\overline{O(y)}^{d},d)$ is proper and by the Całka theorem and nonexpansivity of $f$ we get  $d(f^n(y),y)\rightarrow \infty$, $n\rightarrow \infty. $  
	By \cite[Observation 3.1]{Ka1}, there is a sequence $\{n_i\}$ such that $d_{m}<d_{n_i}$ for $m<n_i, \ i=1,2,\ldots$. By Axiom $1'$ and since $\overline{f(D)}^{||\cdot||}$ is compact  (going to another subsequence if necessary) we can assume that $\{f^{n_i}(y)\}$ converges to some $\xi \in \partial D$. Fix $x \in D$ and choose a subsequence  $\{f^{a_{k}}(x)\}$ of $\{f^n(x)\}$ which converges to some $\eta \in \partial D$. Fix $k \in \mathbb{N}$. Then for sufficiently large $n_i$, we get  	
	\begin{eqnarray*}
		& & d(f^{a_k}(x),f^{n_i}(y)) - d(f^{n_i}(y),y) \\ &\leq& d(f^{a_k}(x),f^{a_k}(y))+ d(f^{a_k}(y),f^{n_i}(y)) - d(f^{n_i}(y),y)	\\ &\leq& d(x,y) + d(f^{n_i -a_k}(y),y) - d(f^{n_i}(y),y) \\ &\leq& 
		d(x,y) + d(f^{n_i}(y),y) - d(f^{n_i}(y),y) \\ &\leq& d(x,y). 
	\end{eqnarray*}	
	Hence there exists a subsequence $\{n_{i_k}\}$ of $\{n_i\}$ such that for any $k$, 
	\begin{equation}\label{ogrGH} 
		d(f^{a_k}(x),f^{n_{i_k}}(y)) - d(f^{n_{i_k}}(y),y) \leq d(x,y). 
	\end{equation}
	Then by (\ref{ogrGH}) we get 
	\begin{eqnarray*} & & \liminf_{k \to \infty} \, d(f^{a_k}(x),f^{n_{i_k}}(y)) - \max \{d(f^{a_k}(x),y),d(f^{n_i}(y),y) \} \\ &\leq& \liminf_{k \to \infty} \, d(f^{a_k}(x),f^{n_{i_k}}(y)) - d(f^{n_{i_k}}(y),y) \\ &\leq& d(x,y).
	\end{eqnarray*}
	It follows from Axiom $2^{\ast}$ that $[\eta, \xi] \subset \partial D$ and hence $\eta \in \ch(\xi)$.
\end{proof}

\section{Attractor of a semigroup of nonexpansive mappings}

The objective of this section is to extend the results of Section 4 to the case of continuous one-parameter semigroups of nonexpansive mappings. \\

A family $S=\{f_{t}:Y\rightarrow Y\mid t\in \lbrack 0,\infty )\}$ is called
a \textit{one-parameter continuous semigroup} if for all $s,t\in \lbrack
0,\infty )$, 
\begin{equation*}
	f_{s+t}=f_{t}\circ f_{s},
\end{equation*}%
and for every $y\in Y$ there is a limit $$\lim_{t\rightarrow
	0^{+}}f_{t}(y)=f_{0}(y)=y.$$

Let's mark it with a symbol $\omega_{S}(x)$ the set of accumulation points (in norm topology) of a semigroup $S$ defined as	
\begin{equation*}
	\omega_{S} (x) = \{y \in \overline{D}: ||f_{t_n}(x)-y|| \rightarrow 0 \; \; \mbox{for some increasing sequence \;} \{t_n\} \rightarrow \infty \}.
\end{equation*}
The \textit{attractor} of the semigroup $S$ is the set $\Omega_{S}$ defined as 
\begin{equation*}
	\Omega_{S} = \bigcup_{x \in D} \omega_{S} (x).
\end{equation*}

The next lemma says that the attractor of the semigroup $S$ is the same set as the attractor of the mapping $f_{t_0}$ for some $t_0>0$.

\begin{lemma}
	\label{attractors}	
	Let $D$ be a bounded convex domain in a Banach space $V$, and let $(D,d)$ be a $(1,\kappa)$-quasi geodesic space satisfying Axiom $5'$ whose topology coincide with the norm topology. Suppose that $S = \{f_t:D \rightarrow D\}_{t \geq 0}$ is a one-parameter continuous semigroup   of nonexpansive mappings without bounded orbits, and there exists $t_0$ such that $f_{t_0}:D \rightarrow D$ is a compact mapping. Then for every $t_0>0$, 
	\begin{equation*}
		\Omega_{S} = \Omega_{f_{t_0}}.
	\end{equation*}
\end{lemma}

\begin{proof}
	Fix $t_0>0$.\\
	$"\supset"$  
 We consider $\eta \in \Omega_{f_{t_0}}$. It follows from the definition of the attractor that there exist $x_0$ and a subsequence $\{f^{a_n}_{t_0}(x_0)\}$ of $\{f^{n}_{t_0}(x_0)\}$ such that $f^{a_n}_{t_0}(x_0) \rightarrow \eta$, as $n \to \infty$. Note that 
 $f^n_{t_0}(x) = f_{nt_0}(x)$. Then $f_{nt_0}(x) \rightarrow \eta$, as $n \to \infty$. Hence $\eta \in \Omega_{S}$.  \\
	$"\subset"$ To show the opposite inclusion, choose $\eta \in \Omega_{S}$. Again from the definition of the attractor, there is a monotone sequence $t_1<t_2< \ldots, \; $  $t_n \rightarrow \infty$ and $x_0 \in D$ such that
	\begin{equation}\label{zbwnorm}
		||f_{t_n}(x_0) - \eta|| \rightarrow 0,
	\end{equation}
as $n \to \infty$. For any $n \in \N$, there exist ${a_n} \in \N$ and $s_n \in [0,t_0)$ such that $t_n = a_nt_0+s_n$. Furthermore, $a_1 \leq a_2 \leq \ldots$. We can assume (considering a subsequence of this sequence) that $a_1<a_2<\ldots$. Hence 
	\begin{equation*}
		f_{t_n}(x_0) = f_{a_nt_0+s_n}(x_0) = f_{t_0}^{a_n}(f_{s_n}(x_0)).
	\end{equation*}
We can suppose that $s_n \rightarrow s_0 \in [0,t_0]$. By continuity of $S$ we get
	\begin{equation}\label{cgdoatr} 
		d(f_{s_n}(x_0),f_{s_0}(x_0)) \rightarrow 0,
	\end{equation}
	as $n \to \infty$.	Hence  
	\begin{equation*}
		d(f_{t_0}^{a_n}(f_{s_n}(x_0)),f_{t_0}^{a_n}(f_{s_0}(x_0))) \leq d(f_{s_n}(x_0),f_{s_0}(x_0)) \rightarrow 0, 
	\end{equation*}
	as $n \to \infty$. Since by assumption, Axiom $5'$ is satisfied, we have
	\begin{equation}\label{zbza5} 
		||f_{t_0}^{a_n}(f_{s_n}(x_0))-f_{t_0}^{a_n}(f_{s_0}(x_0)) || \rightarrow 0,
	\end{equation}
	as $n \to \infty$. Then it follows from (\ref{zbwnorm}) and (\ref{zbza5}) that
	\begin{eqnarray*}
		||f_{t_0}^{a_n}(f_{s_0}(x_0)) - \eta || &\leq& ||f_{t_0}^{a_n}(f_{s_0}(x_0)) - f_{t_0}^{a_n}(f_{s_n}(x_0))|| + ||f_{t_0}^{a_n}(f_{s_n}(x_0)) - \eta|| \\
		&\leq & ||f_{t_0}^{a_n}(f_{s_0}(x_0)) - f_{t_0}^{a_n}(f_{s_n}(x_0))||+||f_{t_n}(x_0)- \eta|| \rightarrow 0,
	\end{eqnarray*}	
	as $n \to \infty$. Therefore, 
	\begin{equation*}
		\eta \in \omega_{f_{t_0}}(f_{s_0}(x_0)) \subset \Omega_{f_{t_0}}.
	\end{equation*}
\end{proof}

The above lemma in combination with Theorem \ref{chuog} gives the following result, which is the counterpart of Theorem \ref{chuog} for one-parameter continuous semigroups.

\begin{theorem} 
	\label{chuogS} Let $D$ be a bounded convex domain in a Banach space $V$, and let $(D,d)$ be a $(1,\kappa)$-quasi geodesic space satisfying Axiom $1'$, Axiom $5'$ and condition $(\ref{C})$ whose topology coincide with the norm topology.  If $S=\{ f_t:D \rightarrow D\}_{t \geq 0}$ is a one-parameter continuous semigroup   of nonexpansive mappings without bounded orbits, and there exists $t_0$ such that $f_{t_0}:D \rightarrow D$ is a compact mapping. Then there exists $\xi \in \partial D$ such that 
	\begin{equation*}
		\Omega_S \subset \ch  (\bigcap_{r\in \mathbb{R}} \, \overline{F_{z_0}(\xi,r)}^{||\cdot||})
	\end{equation*}
	for some $z_0 \in D$.
\end{theorem}

Using again Lemma \ref{attractors}, combined with Theorem \ref{chchdys} or Theorem \ref{gwHil} respectively, this allows us to obtain the following two results.

\begin{theorem}
	Let $D$ be a bounded convex domain in a Banach space $V$, and let $(D,d)$ be a $(1,\kappa)$-quasi geodesic space satisfying Axiom $1'$, Axiom $5'$ and condition (\ref{C}) whose topology coincide with the norm topology. If $S=\{ f_t:D \rightarrow D\}_{t \geq 0}$ is a one-parameter continuous semigroup of nonexpansive mappings without bounded orbits, and there exists $t_0$ such that $f_{t_0}:D \rightarrow D$ is a compact mapping, then there exists $\xi \in \partial D$ such that 
	\begin{equation*}
		\Omega_S \subset \ch (\ch(\xi)).
	\end{equation*}
\end{theorem}

\begin{theorem} 
	Let $D$ be a bounded convex domain in a Banach space $V$, and let $(D,d)$ be a $(1,\kappa)$-quasi geodesic space satisfying  Axiom $1'$, Axiom $2^{\ast}$ and Axiom $5'$ whose topology coincide with the norm topology. If $S=\{ f_t:D \rightarrow D\}_{t \geq 0}$ is a one-parameter continuous semigroup of nonexpansive mappings without bounded orbits, then there exists $\xi \in \partial D$ such that
	\begin{equation*}
		\Omega_S \subset \ch(\xi).
	\end{equation*}
\end{theorem}

\bigskip 

\textbf{Acknowledgements}  Author  was partially supported by National Science Center (Poland) Preludium Grant No. UMO-2021/41/N/ST1/02968.


\bigskip
\begin{thebibliography}{99}

\bibitem{Ab} M. Abate, Horospheres and iterates of holomorphic maps, Math. Z. 198 (1988), 225--238.

\bibitem{AbRa} M. Abate and J. Raissy, Wolff--Denjoy theorems in nonsmooth convex domains, Ann. Mat. Pura Appl. 193 (2014), 1503--1518.

\bibitem{Be1} A. F. Beardon, Iteration of contractions and analytic maps, J. London Math. Soc. 41 (1990), 141--150.

\bibitem{Be2} A. F. Beardon, The dynamics of contractions, Ergod. Th. \& Dynam. Sys. 17 (1997), 1257--1266.

\bibitem{Bu1} M. Budzy\'{n}ska, A Denjoy--Wolff theorem in $C^{n}$, Nonlinear Anal. 75 (2012), 22--29.

\bibitem{Ha} L. A. Harris, Schwarz-Pick systems of pseudometrics for domains in normed linear spaces,  Advances in holomorphy, J. A. Barroso (ed.), North Holland (1979), 345-406.

\bibitem{HuWi1} A. Huczek, A. Wi\'snicki, Wolff–Denjoy theorems in geodesic spaces, Bull. Lond. Math. Soc. 53 (2021), 1139-1158.

\bibitem{HuWi2} A. Huczek, A. Wi\'snicki, Theorems of Wolff–Denjoy type for semigroups of nonexpansive
mappings in geodesic spaces,  Math. Nachr. (2023). 

\bibitem{Ka1} A. Karlsson, Nonexpanding maps and Busemann functions, Ergodic Theory Dynam. Systems 21, (2001) 1447--57.

\bibitem{Ka3} A. Karlsson, Dynamics of Hilbert nonexpansive maps, in Handbook of Hilbert geometry, A. Papadopoulos and M. Troyanov (eds.), European Mathematical Society, Z\"{u}rich, 2014, pp. 263--273.

\bibitem{KaNo} A. Karlsson and G.A. Noskov, The Hilbert metric and Gromov hyperbolicity, Enseign. Math. (2) 48 (2002), 73--89.

\bibitem{KRS} T. Kuczumow, S. Reich, D. Shoikhet,  Fixed points of holomorphic mappings: a metric approach, in
Handbook of Metric Fixed Point Theory, W. A. Kirk and B. Sims (eds.), {\em Kluwer Academic Publishers}, Dordrecht, (2001), 437-515.


\bibitem{LLNW} B. Lemmens at al., Denjoy-Wolff theorems for Hilbert's and Thompson's metric spaces, J. Anal. Math. 134 (2018), 671--718.

\bibitem{LeNu} M. Lemmens and R. Nussbaum, Nonlinear Perron--Frobenius theory, Cambridge University Press, Cambridge, 2012.

\bibitem{Nu} R. D. Nussbaum, Fixed point theorems and Denjoy--Wolff theorems for Hilbert's projective metric in infinite dimensions, Topol. Methods Nonlinear Anal. 29 (2007), 199--250.

\bibitem{Pi} B. Pi\c{a}tek, The behavior of fixed point free nonexpansive mappings in geodesic spaces, J. Math. Anal. Appl. 445 (2017), 1071–1083.

\end{thebibliography}
\end{document}